\documentclass[a4paper,oneside,10pt]{amsart}
\usepackage{bm}
\usepackage[pdftex]{graphicx,xcolor}
\usepackage[pdftex]{hyperref}
\usepackage{ascmac}
\usepackage{amsmath}
\usepackage{amscd}
\usepackage{amssymb}
\usepackage{amsfonts}
\usepackage{amsthm}
\usepackage[all]{xy}
\usepackage{mathtools}
\usepackage{enumerate}
\usepackage{tikz}
\usepackage{hyperref}

\newtheorem{thm}{Theorem}

\newtheorem{lem}[thm]{Lemma}
\newtheorem{prop}[thm]{Proposition}

\newcommand{\ES}{\operatorname{ES}}

\title{Big convex polytopes or rich hyperplanes}
\author{Koki Furukawa}


\begin{document}

\begin{abstract}
We extend the famous Erd\"{o}s-Szekeres convex polygon problem to arbitrary point sets in $\mathbb{R}^d$.
For $n, l> d \geq 2$, let $\ES_d (l,n)$ be the smallest integer $N$ such that any set of at least $N$ points in $\mathbb{R}^d$ contains either $l$ points contained in a common $(d-1)$-dimensional hyperplane or $n$ points in convex position.
In this paper, we give the upper and lower bounds for $\ES_d (l,n)$.
\end{abstract}

\maketitle
\section{Introduction}
\subsection{Erd\"{o}s-Szekeres theorem}
For $d \geq 2$, we say that a finite point set $P$ with $|P| \geq d+1$ in $\mathbb{R}^d$ is in \textit{general position} if no $(d+1)$ points lie on a common $(d-1)$-dimensional hyperplane.
We say that $P$ is in \textit{convex position} if every point of $P$ is a vertex of $conv(P)$ and they are in general position. 
For $n > d \geq 2$, let $\ES_d (n)$ be the smallest integer $N$ such that any set of (at least) $N$ points in $\mathbb{R}^d$ contains $n$ points in convex position.
In 1935, Erd\"{o}s and Szekeres \cite{ES} proved $\ES_2 (n) \leq \binom{2n-4}{n-2} + 1$.
In 1960, they \cite{ES2} showed $\ES_2 (n) \geq 2^{n-2} +1$ by constructing $2^{n-2}$ points in the plane containing no $n$ points in convex position and this bound is believed to be optimal.
The upper bound was not improved for 60 years. 
In 1998, Graham and Chung \cite{GC} improved their bound $\binom{2n-4}{n-2} + 1$ by $1$.
Shortly after, Kleitman and Pachter \cite{KP} showed that $\ES_2 (n) \leq \binom{2n-4}{n-2} +7 - 2n$.
%どこまで書くか
There have been only small improvements of this upper bound. 
In 2017, Suk \cite{SUK} gave a great improvement which is $\ES_2 (n) \leq 2^{n + O(n^{\frac{2}{3}} \log n)}$. 
Shortly after, Holmsen, Mojarrad, Pach and Tardos \cite{HOLMSEN} showed that 
$\ES_2 (n) \leq 2^{n + O(\sqrt{n \log n})}$ by optimizing Suk's argument. See \cite{CC} for a clear proof.
Currently, this is the best known upper bound.

\subsection{Erd\"{o}s-Szekeres theorem for higher dimension}
Various extensions of the result have been investigated.
Erd\"{o}s and Szekeres also noted in their 1935 paper \cite{ES}, that the number $\ES_d (n)$ is finite for all $n > d \geq 3$.
they gave an upper bound $\ES_d (n) \leq R_{d+2} (n, d+3)$ where $R_{d+2} (n, d+3)$ is called the hypergraph Ramsey number which is the minimum $N$ such that every red-blue colouring of the $(d+2)$-tuples of an $N$-element set contains a red set of size $d+3$ or a blue set of size $n$. The hypergraph Ramsey number is known to be enormously large. 
We know that there is a constant $c = c (d) > 0$ such that 
$R_{d+2} (n, d+3) > twr_d (n^{c \log n})$ where the \textit{tower function} $twr_d (x)$ is defined by $twr_1 (x) = x$ and $twr_{i+1} (x) = 2^{twr_i (x)}$. 
See \cite{MS} for a detailed discussion.
Valtr \cite{Valtr} found another way to prove the existence of $\ES_d (n)$.
Consider any set $X$ of at least $\ES_{d-1} (n)$ points in general position in $\mathbb{R}^d$ 
and its projection $Y$ onto a generic $(d-1)$-dimensional hyperplane.
Since $|Y| \geq \ES_{d-1} (n)$, there is a subset $Z$ of size $n$ in convex position.
He pointed out that the set obtained by lifting $Z$ back to former space $\mathbb{R}^d$ is also in convex position. This implies that 
$$
\ES_d (n) \leq \ES_{d-1} (n) \leq \cdots \leq \ES_2 (n) (\leq 2^{n + O(\sqrt{n \log n})}).
$$
Moreover, K\'{a}rolyi \cite{Karolyi} proved that $\ES_d (n) \leq \ES_{d-1} (n-1) + 1$, and this implies $\ES_d (n) \leq \binom{2n - 2d - 1}{n - d} + 1$.
A recent result of Pohoata and Zakharov \cite{PZ} shows that $\ES_3 (n) = 2^{o(n)}$. 
Their $o(n)$ takes the form $\frac{n}{\log_{(5)} n}$ where $\log_{(k)} n$ is the $k$-th iterated logarithm function. 
The idea of their proof is notable for its continued reliance on techniques from the two-dimensional case. The key approach by them involves a projection-and-lifting technique, where they first project the point configurations from $\mathbb{R}^3$ onto $\mathbb{R}^2$, perform analysis similar to the one discussed by Suk \cite{SUK}, and then lift the results back to $\mathbb{R}^3$. 
On the other hand, K\'{a}rolyi and Valtr \cite{KV} showed that there is a constant $c = c(d) > 1$ such that
$\ES_d (n) \geq 2^{c n^{\frac{1}{d-1}}}$ for every $d \geq 2$ and this bound is believed to be optimal.

\subsection{Erd\"{o}s-Szekeres theorem for arbitrary point sets}
Let $\ES (l, n)$ be the minimum $N$ such that every $N$-element point set in the plane contians either $l$ collinear points or $n$ points in convex position. 
In 2024, Conlon et al. \cite{BLBC} proved that there is a constant $C > 0$ such that for each $l, n \geq 3$, 
$$
(3l-1) \cdot 2^{n-5} < \ES (l, n) < l^2 \cdot 2^{n + C \sqrt{n \log n}}.
$$

It is natural to consider the generalization of this problem to higher dimensions, that is, to find $n$ points in convex position in a sufficiently large point sets that contains no $l$ points on the same hyperplane.
$n,l> d \geq 2$, let $\ES_d (l,n)$ be the smallest integer $N$ such that any set of at least $N$ points in $\mathbb{R}^d$ contains either $l$ points contained in a common $(d-1)$-dimensional hyperplane or $n$ points in convex position.
Therefore, $\ES_2 (l, n) = \ES (l, n)$ holds.
We prove the following upper and lower bounds for $\ES_d (l, n)$.

\begin{thm} \label{thm1.1}
For $d \geq 3$,
$$
\ES_d (l, n) \leq \frac{l-d}{d!} \cdot 2^{O(\frac{dn}{\log_{(4)} n})}.
$$
\end{thm}

Generalizing the argument in \cite{KV}, we obtain the following lower bounds.

\begin{thm} \label{thm1.2}
There is a constant $c = c_d > 1$ such that 
if $d$ is even, 
$$
\ES_d (l,n) \geq \frac{l-1}{d} \cdot 2^{c_d n^{\frac{1}{d-1}}},
$$
if $d$ is odd, 
$$
\ES_d (l,n) \geq \frac{l-2}{d-1} \cdot 2^{c_d n^{\frac{1}{d-1}}}.
$$

\end{thm}
Here the order of the lower bound is roughly 
$$
\ES_d (l,n) \geq \frac{l - d + 2 \lfloor \frac{d}{2} \rfloor - 1}{2 \lfloor \frac{d}{2} \rfloor} \cdot 2^{c_d n^{\frac{1}{d-1}}} \gtrsim \frac{l-1}{d} \cdot 2^{c_d n^{\frac{1}{d-1}}}.
$$

\section{Upper bound} 
As mentioned above, $\ES_d (n) \le \ES_{d-1}(n)$ holds. Similarly, one might think that $\ES_d (l, n) \le \ES_{d-1} (l, n)$ holds by applying projection arguments. 

However, this is not always true. This is because when projecting a $d$-dimensional point configuration onto $(d-1)$-dimensional space and then trying to reconstruct the $d$-dimensional configuration, $n$-points in convex position in $(d-1)$-dimensions are not necessarily general position when projected back to $d$-dimensions.
It should be noted once again that, in order for a set of points to be in convex position, it must be in general position.
We did not find any proof using projection techniques to establish this type of inequality. Our approach is as follows:

Let $X$ be a set of points in $\mathbb{R}^d$ contains no $m$ points in general position and no $l$ points on the same $(d-1)$-dimensional hyperplane. To give the upper bound for $\ES_d (l, n)$, it suffices to estimate the size of $X$.
\begin{proof}[Proof of Theorem \ref{thm1.1}] 
Let  $S (\subseteq X)$ be a maximal set of points in general position.
Then, each point in $X \backslash S$ lies on a $(d-1)$-dimensional hyperplane determined by some $d$ points in $S$. 
Otherwise, there exists a subset of $X$ in general position that has a size of bigger than $|S|$. 
By choosing $d$ points from $S$, the set $S$ defines $\binom{|S|}{d}$ hyperplanes and each hyperplane contains at most $l-(d+1)$ points in $X \backslash S$. This implies 
$|X| \leq (l-d-1) \cdot \binom{|S|}{d} + |S| \leq (l-d-1) \cdot \binom{m-1}{d} + m - 1$.
Therefore, every set of at least $(l-d-1) \cdot \binom{m-1}{d} + m$ points in $\mathbb{R}^d$ contains either $l$ points on the same $(d-1)$-dimensional hyperplane or $m$ points in general position. 
Then, it immediately follows that
$$
\ES_d (l,n) \leq (l-d-1) \cdot \binom{\ES_d (n)-1}{d} + \ES_d (n).
$$
Together with the upper bound $\ES_d (n) = 2^{o(n)}$ by Pohoata and Zakharov \cite{PZ}, the conclusion follows.

\end{proof}

For $d=2$, the same approach can be used to obtain the following upper bound;
$$
\ES_2 (l,n) \leq (l-3) \cdot \binom{\ES_2 (n)-1}{2} + \ES_2 (n) \lesssim l \cdot  4^{n + O(\sqrt{n\log n})}
$$
However, this bound is weaker than the one given by Conlon et al. unless $l$ is significantly larger than $n$.

\section{Lower bound}
As we mentioned above, Kalolyi and Valtr \cite{KV} proved $\ES_d (n) \geq 2^{c_d n^{\frac{1}{d-1}}}$ for some $c_d > 1$.  
Their construction is as follows: 
Start with one point set $X_0$, and $X_{i+1}$ is obtained from $X_i$ by replacing each point $x \in X_i$ with two points $x - v(x)$ and $x + v(x)$ where $v(x) = (v^1 (x), \ldots, v^d (x))$ is a vector which satisfies $0 < v^1 (x) < v^2 (x) < \cdots < v^d (x) < \varepsilon_i$ and $v^f (x) < \varepsilon_i v^{f+1} (x)$ for every $1 \leq f \leq d - 1$ and $\varepsilon_i > 0$ be sufficiently small.
Then, apply a small perturbation (i.e. taking $\varepsilon_i$ sufficiently small) to $X_{i+1}$ to be in general position.
By construction, we have $|X_i| = 2^i$, and the key lemma is that the inequality
\begin{equation}\label{eq:eq01}
mc (X_{i+1}) \leq mc (X_i) + mc (\pi_{d-1} (X_i))
\end{equation}
holds where $mc(X) \coloneqq \textrm{max} \{ |S| : S (\subseteq X) \, \mbox{is in convex position} \}$ and $\pi_{d-1} : \mathbb{R}^d \rightarrow \mathbb{R}^{d-1}$ is the projection to the $(d-1)$-dimensional hyperplane $\{ x_d = 0 \}$.

\begin{figure}[htbp]
\tikzset{every picture/.style={line width=0.75pt}} %set default line width to 0.75pt        


\caption{}
\end{figure}

To prove Theorem 2, we consider the set constructed by replacing many points on the line segment $conv(x-v(x), x+v(x))$ for each $i \geq 1$ and each $x \in X_i$.
The details are as follows.
For a point $x \in X$, we define the following set  %$d$がevenのとき，$\lfloor \frac{2(l-1)}{d} \rfloor$点,$d$がoddのとき，$\lfloor \frac{2(l-2)}{d-1} \rfloor$点からなる$conv(x-v(x), x+v(x))$の部分集合として定める．
$$
P_x (\varepsilon, d, l) \coloneqq 
\begin{cases}
\{ x \pm \frac{1}{k} v(x) : k=1, \ldots, \frac{1}{2} \lfloor \frac{2(l-1)}{d} \rfloor \} & (d : \mbox{even},\, l : \mbox{odd}) \\
\{ x \pm \frac{1}{k} v(x) : k=1, \ldots, \frac{1}{2} \lfloor \frac{2(l-1)}{d} \rfloor - \frac{1}{2} \} \cup \{ x \} & (d : \mbox{even},\, l : \mbox{even}) \\
\{ x \pm \frac{1}{k} v(x) : k=1, \ldots, \frac{1}{2} \lfloor \frac{2(l-2)}{d-1} \rfloor \} & (d : \mbox{odd},\, l : \mbox{even}) \\
\{ x \pm \frac{1}{k} v(x) : k=1, \ldots, \frac{1}{2} \lfloor \frac{2(l-3)}{d-1} \rfloor - \frac{1}{2} \} \cup \{ x \} & (d : \mbox{odd},\, l : \mbox{odd}).
\end{cases}
$$ 
and define $Y_{i+1}$ as the set obtained by replacing each point $x$ in $X_i$ with $P_x (\varepsilon_i, d, l)$. 
See Figure 1.
If it is not a cause for ambiguity, we will denote $P_x \coloneqq P_x (\varepsilon, d, l)$.

We say that two finite sets $\{ p_1, \ldots, p_n \} \subset \mathbb{R}^d$ and $\{ q_1, \ldots, q_n \} \subset \mathbb{R}^d$ have the $\textit{same order type}$ if the orientations of $p_{i_1} \cdots p_{i_{d+1}}$ and $q_{i_1} \cdots q_{i_{d+1}}$ are the same for all $1 \leq i_1 \leq, \ldots, i_{d+1} \leq n$, that is, 
$$
sgn\bigg(det\begin{pmatrix}
p_{i_1} & p_{i_2} & \cdots & p_{i_{d+1}} \\
1 & 1 & \cdots & 1
\end{pmatrix}\bigg) 
= 
sgn\bigg(det\begin{pmatrix}
q_{i_1} & q_{i_2} & \cdots & q_{i_{d+1}} \\
1 & 1 & \cdots & 1
\end{pmatrix}\bigg)
\in \{ \pm 1, 0 \}.
$$

\begin{prop}
For each $i \geq 1$, the set $Y_i$ contains no $l$ points on the same $(d-1)$-dimensional hyperplane.

\end{prop}

\begin{proof}
Set $X_{i-1} = \{ x_1, \ldots, x_{2^{i-1}} \}$ and we pick a point $p_{x_k}$ from each $P_{x_k} \,\,\,\,\, (k \in \{ 1, \ldots, 2^{i-1} \})$.
Note that since a subset $\{ p_{x_1}, \ldots, p_{x_{2^{i-1}}} \}$ of $Y_i = \bigcup_{k=1}^{2^{i-1}} P_{x_k}$ and $X_{i-1}$ has the same order type, and $X_{i-1}$ is in general position, every hyperplane contain at most $d+1$ distinct $P_{x_k}$'s. 
The affine hull of each $P_{x_k}$ is a line and they are all in a skew position. Therefore, if $d$ is even, the affine hull of the union of any $\frac{d}{2}$ $P_{x_k}$'s is a $(d-1)$-flat (i.e., a hyperplane). This $(d-1)$-flat contains $\frac{d}{2} \lfloor \frac{2}{d} (l-1) \rfloor \leq l-1$ points from $Y_i$. If $d$ is odd, the union of any $\frac{d-1}{2}$ $P_{x_k}$'s forms a $(d-2)$-flat. By taking an additional point from one of the other $P_{x_k}$, their affine hull is a hyperplane. This hyperplane contains $\frac{d-1}{2} \lfloor \frac{2}{d-1} (l-2) + 1 \rfloor \leq l-1$ points from $Y_i$. See Figure 2.
\end{proof}

\begin{figure}[htbp]

\tikzset{every picture/.style={line width=0.75pt}} %set default line width to 0.75pt        



\caption{}

\end{figure}

\begin{lem}
For each $i \geq 1$,
$$mc (Y_i) = mc (X_i).$$
\end{lem}

\begin{proof}
Since $X_i$ is a subset of $Y_i$ for each $i \geq 1$, $mc(X_i) \leq mc(Y_i)$ holds.

For simplicity, we denote $Y_i$ by $\bigcup_{k=1}^{2^{i-1}} P_{x_k}$. 
Let $C$ be an arbitrary subset of $Y_i$ in convex position. 
By the definition of convex position, we have $|C \cap P_{x_k}| \leq 2$ for every $k \in \{1, \dots, 2^{i-1}\}$. 
For such a subset $C$, consider the subset $C' \subset X_i$ defined by
\[
    C' \coloneqq \bigcup_{1 \leq |C \cap P_{x_k}| \leq 2} \partial(\mathrm{conv}(P_{x_k})).
\]
It is easy to see that $C'$ is in convex position and satisfies $|C| \leq |C'|$. See Figure 3.
\end{proof}

\begin{figure}[htbp]

\tikzset{every picture/.style={line width=0.75pt}} %set default line width to 0.75pt        



\caption{}
\end{figure}

\begin{proof}[Proof of Theorem \ref{thm1.2}]

Thus, from \eqref{eq:eq01} and Lemma 4, we obtain $mc(Y_i) \leq 2 i^{d-1}$ by double induction.
Finally, simple counting shows that
$$
\begin{aligned}
|Y_i| &=
\begin{cases}
\lfloor \frac{2(l-1)}{d} \rfloor \cdot 2^{i-1} & (d : \mbox{even}) \\
\lfloor \frac{2(l-2)}{d-1} \rfloor \cdot 2^{i-1} & (d : \mbox{odd})
\end{cases} \\
&\geq 
\begin{cases}
\frac{l-1}{d} \cdot 2^i & (d : \mbox{even}) \\
\frac{l-2}{d-1} \cdot 2^i & (d : \mbox{odd})
\end{cases} \\
&= 
\frac{l - d + 2 \lfloor \frac{d}{2} \rfloor - 1}{2 \lfloor \frac{d}{2} \rfloor} \cdot 2^i
\end{aligned}
$$

Therefore, we can conclude that there exists a set of $\frac{l - d + 2 \lfloor \frac{d}{2} \rfloor - 1}{2 \lfloor \frac{d}{2} \rfloor} \cdot 2^{c_d n^{\frac{1}{d-1}}}$ points in $\mathbb{R}^d$ with no $l$ points on the same $(d-1)$-dimensional hyperplane and no $n$ points in convex position, which proves Theorem 2.

\end{proof}
\,\\
\textbf{Remark.}
K\'{a}rolyi and Valtr also showed that $\textrm{mc} (X_i) \leq \frac{2}{(d-1)!} i^{d-1} + O(i^{d-2})$. See Appendix in \cite{KV}. This gives $c_d \approx 2^{e^{-1} d} \approx 2^{0.37 d}$.

\section{Concluding Remarks}
We proved that $\frac{l-1}{d} \cdot 2^{c_d n^{\frac{1}{d-1}}} \lesssim \ES_d (l, n) \leq \frac{l-d}{d!} \cdot 2^{O(\frac{dn}{\log_{(4)} n})}$. 
We believe that Theorem 2 is optimal for all $d > 2$, except for the exact value of the constant $c_d$.
It might be useful to determine the following value  $C_d (l, n)$ to give a better upper bound: 
$C_d (l,n)$ is the minimum $N$ such that every set of $N$-points in $\mathbb{R}^d$ in weakly convex position contains either $l$ points on the same $(d-1)$-dimensional hyperplane or $n$ in convex position 
where a set of points $P (\subset \mathbb{R}^d)$ is said to be in \textit{weakly convex position} if $P \subseteq \partial (conv(P))$. 
We believe that $C_d (l,n) \approx \frac{nl}{d}$ and if this is shown, it immediately follows that $\ES_d (l, n) \approx \ES_d (\frac{nl}{d})$. 
However, even if this is possible, it remains asymptotically far from the lower bound.

\,\\
\centerline{{\nolinkurl{aluvajp@gmail.com}}}
\end{document}